\documentclass[11pt,a4paper]{amsart}
\pdfoutput=1
\usepackage{hyperref,amsthm,amsfonts,amsmath,float,graphicx}
\usepackage[numbers,sort&compress]{natbib}
\usepackage[lmargin=34mm,rmargin=34mm,tmargin=34mm,bmargin=34mm]{geometry}

\theoremstyle{plain}
\newtheorem{theorem}{Theorem}
\newtheorem{proposition}[theorem]{Proposition}
\newtheorem{conjecture}[theorem]{Conjecture}
\newtheorem{lemma}[theorem]{Lemma}
\newtheorem{corollary}[theorem]{Corollary}
\newtheorem{observation}[theorem]{Observation}

\newcommand{\comment}[1]{\medskip\noindent\framebox{\parbox{\textwidth}{\smallskip{#1}\smallskip}}}

\newcommand{\figlabel}[1]{\label{fig:#1}}

\newcommand{\seclabel}[1]{\label{sec:#1}}

\newcommand{\spacing}[1]{\renewcommand{\baselinestretch}{#1}\setlength{\footnotesep}{\baselinestretch\footnotesep}}
\spacing{1.14}

\newcommand{\B}{\mathcal{B}}
\newcommand{\C}{\mathcal{C}}

\newcommand{\half}{\ensuremath{\protect\tfrac{1}{2}}}
\newcommand{\ceil}[1]{\ensuremath{\protect\lceil#1\rceil}}

\newcommand{\floor}[1]{\ensuremath{\protect\lfloor#1\rfloor}}

\DeclareMathOperator{\conv}{conv}

\begin{document}

\title[On the Connectivity of Visibility Graphs]{On the Connectivity of Visibility Graphs}

\date{\today}

\author[]{Michael~S.~Payne}
\address{Department of Mathematics and Statistics, 
The University of Melbourne
\newline Melbourne, Australia}
\email{m.payne3@pgrad.unimelb.edu.au}

\author[]{Attila P\'or}
\address{Department of Mathematics, 
 Western Kentucky University
\newline Bowling Green,  Kentucky, U.S.A.}
\email{attila.por@wku.edu}

\author[]{Pavel Valtr}
\address{Department of Applied Mathematics, 
Charles University
\newline Prague, Czech Republic}
\email{valtr@kam.mff.cuni.cz}

\author[]{David~R.~Wood}
\address{Department of Mathematics and Statistics, 
The University of Melbourne
\newline Melbourne, Australia}
\email{woodd@unimelb.edu.au}

\begin{abstract}
The visibility graph of a finite set of points in the plane has the points as vertices and an edge between two vertices if the line segment between them contains no other points. 
This paper establishes bounds on the edge- and vertex-connectivity of visibility graphs. 

Unless all its vertices are collinear, a visibility graph has diameter at most $2$, and so it follows by a result of Plesn\'ik (1975) that its edge-connectivity equals its minimum degree. 
We strengthen the result of Plesn\'ik by showing that for any two vertices $v$ and $w$ in a graph of diameter $2$, if $\deg(v) \leq \deg(w)$ then there exist $\deg(v)$ edge-disjoint $vw$-paths of length at most $4$.
Furthermore, we find that in visibility graphs every minimum edge cut is the set of edges incident to a vertex of minimum degree.

For vertex-connectivity, we prove that every visibility graph with $n$ vertices and at most $\ell$ collinear vertices has connectivity at least $\frac{n-1}{\ell-1}$, which is tight. 
We also prove the qualitatively stronger result that the vertex-connectivity is at least half the minimum degree.  
Finally, in the case that $\ell =4$ we improve this bound to two thirds of the minimum degree.
\end{abstract}

\thanks{David Wood is supported by a QEII Research Fellowship from the Australian Research Council.}

\thanks{Michael Payne is supported by an Australian Postgraduate Award from the Australian Federal Government.}

\thanks{This work began during the Special Semester on Discrete and Computational Geometry in 2010 at the \'Ecole Polytechnique F\'ed\'erale de Lausanne. We are grateful to the Centre Interfacultaire Bernoulli for hosting and the Swiss National Science Foundation for supporting this event.}


\maketitle

\section{Introduction}
\seclabel{intro}

Let $P$ be a finite set of points in the plane. Two distinct points
$v$ and $w$ in the plane are \emph{visible} with respect to $P$ if no
point in $P$ is in the open line segment $vw$. The
\emph{visibility graph} of $P$ has vertex set $P$, where two vertices are adjacent if and only if they are visible with respect to $P$. So the visibility graph is obtained by drawing lines
through each pair of points in $P$, where two points are adjacent if they
are consecutive on a such a line. Visibility graphs have 
interesting graph-theoretic properties. 
For example,  K\'ara, P\'or and Wood \cite{KPW} showed that $K_4$-free visibility graphs are $3$-colourable.
See \cite{pfender,KPW,porwood,matoudcg} for results and conjectures about the clique and chromatic number of visibility graphs. Further related results can be found in \cite{EmptyPentagon-GC,DPT09}.


The purpose of this  paper is to study the edge- and vertex-connectivity of visibility graphs.
A graph $G$ on at least $k+1$ vertices is $k$-vertex-connected ($k$-edge-connected) if $G$ remains connected whenever fewer than $k$ vertices (edges) are deleted. Menger's theorem says that this is equivalent to the existence of $k$ vertex-disjoint (edge-disjoint) paths between each pair of vertices. 
Let $\kappa(G)$ and $\lambda(G)$ denote the vertex- and
edge-connectivity of a graph $G$. 
Let $\delta(G)$ denote the minimum
degree of $G$. We have $\kappa(G)\leq \lambda(G)\leq \delta(G)$. For these and other basic results in graph theory see a standard text book such as \cite{diestel}.

If a visibility graph $G$ has $n$ vertices, at most $\ell$ of which are collinear, then $\delta(G) \geq \frac{n-1}{\ell-1}$. We will show that both edge- and vertex-connectivity are at least $\frac{n-1}{\ell-1}$ (Theorem \ref{nledge} and Corollary \ref{nlvert}). Since there are visibility graphs with $\delta = \frac{n-1}{\ell-1}$ these lower bounds are best possible.

We will refer to visibility graphs whose vertices are not all collinear as \emph{non-collinear visibility graphs}.
Non-collinear visibility graphs have diameter $2$ \cite{KPW}, and it is known that graphs of diameter $2$ have edge-connectivity equal to their minimum degree \cite{plesnik}.
We strengthen this result to show that if a graph has diameter $2$ then between any two vertices $v$ and $w$ with $\deg(v) \leq \deg(w)$, there are $\deg(v)$ edge-disjoint paths of length at most $4$ (Theorem \ref{4paths}).
We also characterise minimum edge cuts in visibility graphs as the sets of edges incident to a vertex of minimum degree (Theorem \ref{edgecut}).

With regard to vertex-connectivity, our main result is that $\kappa\geq \frac{\delta}{2} +1$ for all non-collinear visibility graphs (Theorem \ref{halfdelta}). 
This bound is qualitatively stronger than the bound $\kappa \geq \frac{n-1}{\ell-1}$ since it is always within a factor of $2$ of being optimal.
In the special case of at most four collinear points, we improve this bound to $\kappa \geq \frac{2\delta+1}{3}$ (Theorem \ref{4line}). 
We conjecture that $\kappa \geq \frac{2\delta+1}{3}$ for all visibility graphs. This bound would be best possible since, for each integer $k$, there is a visibility graph with a vertex cut of size $2k+1$, but minimum degree $\delta = 3k+1$. Therefore the vertex-connectivity is at most $2k+1 =\frac{2\delta+1}{3}$. Figure~\ref{twothirds} shows the case $k=4$.

\begin{figure}[!h]
\includegraphics{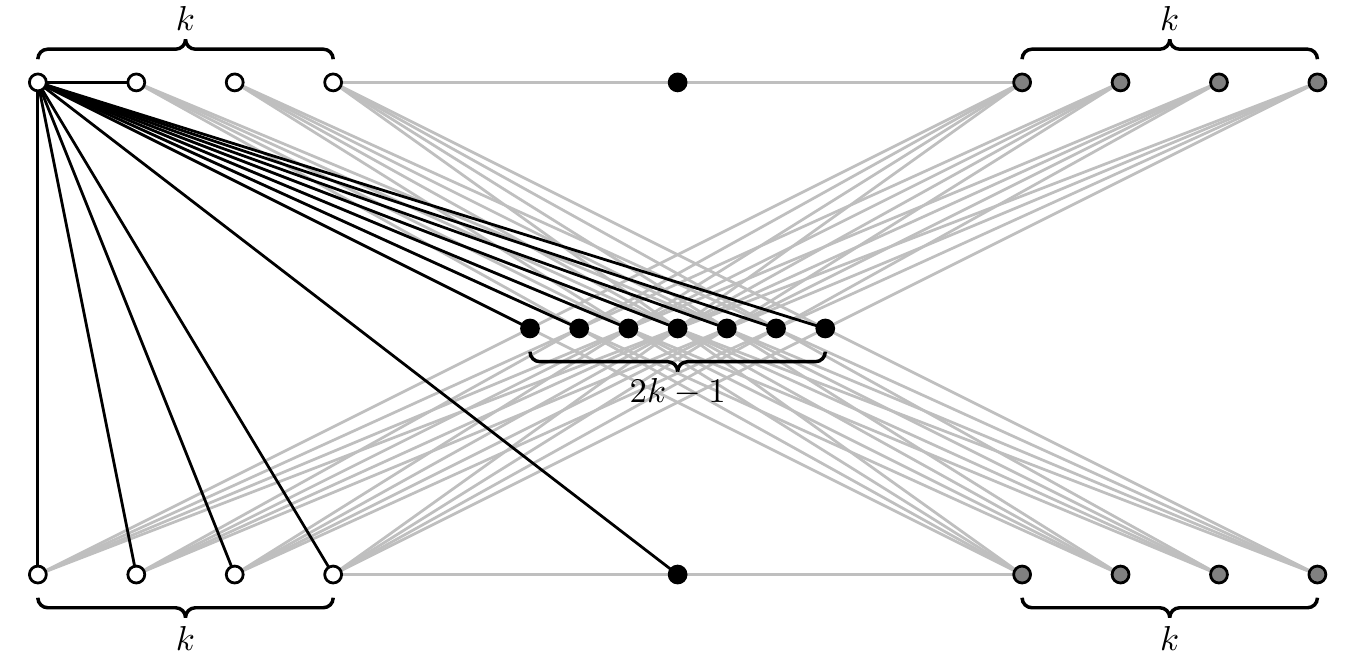}
\caption{\label{twothirds}A visibility graph with vertex-connectivity
$\frac{2\delta+1}{3}$. The black vertices are a cut set. The minimum
degree $\delta = 3k+1$ is achieved, for example, at the top left vertex. Not all edges are drawn.}
\end{figure}

A central tool in this paper, which is of independent interest, is a kind of bipartite visibility graph. 
Let $A$ and $B$ be disjoint sets of points in the plane. 
The \emph{bivisibility graph} $\B(A,B)$ of $A$ and $B$ has vertex set $A\cup B$,
where points $v\in A$ and $w\in B$ are adjacent if and only if they are
visible with respect to $A\cup B$. 
The following simple observation is used several times in this paper.

\begin{observation}\label{obs}
Let $G$ be a visibility graph. Let $\{A,B,C\}$ be a partition of $V(G)$ such that $C$ separates $A$ and $B$.
If $\B(A,B)$ contains $t$ pairwise non-crossing edges, then $\kappa(G) \geq |C| \geq t$
since there must be a distinct vertex in $C$ on each such edge.
%
%
\end{observation}

Finally, one lemma in particular stands out as being of independent interest. Lemma \ref{atlem} says that for any two properly coloured non-crossing geometric graphs that are separated by a line, there exists an edge joining them such that the union is a properly coloured non-crossing geometric graph. 

\section{Edge Connectivity}

Non-collinear visibility graphs have diameter at most $2$ \cite{KPW}. 
This is because even if two points cannot see each other, they can both see the point closest to the line containing them. 
Plesn\'ik \cite{plesnik} proved that the edge-connectivity of a graph with diameter at most $2$ equals its minimum degree. 
Thus the edge-connectivity of a non-collinear visibility graph equals its minimum degree. 
There are several other known conditions that imply that the edge-connectivity of a graph is equal to the minimum degree; see for example \cite{volkmann, dankelvolk, pleznam}.
Here we prove the following
strengthening of the result of Plesn\'ik. 

\begin{theorem}\label{4paths}
Let $G$ be a graph with diameter $2$.
Let $v$ and $w$ be distinct vertices in $G$.
Let $d:=\min\{\deg(v),\deg(w)\}$. 
Then there are $d$ edge-disjoint paths of length at most $4$ between $v$ and $w$ in $G$.
\end{theorem}

\begin{proof}
First suppose that $v$ and $w$ are not adjacent.
Let $C$ be the set of common neighbours of $v$ and $w$.
For each vertex $c\in C$, take the path $(v,c,w)$.
Let $A$ be a set of $d-|C|$ neighbours of $v$ not in $C$.
Let $B$ be a set of $d-|C|$ neighbours of $w$ not in $C$.
Let $M_1$ be a maximal matching in the bipartite subgraph of $G$ induced by $A$ and $B$.
Call these matched vertices $A_1$ and $B_1$.
For each edge $ab\in M_1$, take the path $(v,a,b,w)$.
Let $A_2$ and $B_2$ respectively be the subsets of $A$ and $B$ consisting of the unmatched vertices.
Let $D:=V(G) \setminus (A_2 \cup B_2 \cup \{v,w\})$.
Let $M_2$ be an arbitrary pairing of vertices in $A_2$ and $B_2$.
For each pair $ab \in M_2$, take the path $(v,a,x,b,w)$, where $x$ is a common neighbour of $a$ and $b$ (which exists since $G$ has diameter $2$). 
Since $x$ is adjacent to $a$, $x\neq w$, and by the maximality of $M_1$, $x\not\in B_2$.
Similarly, $x\neq v$ and $x\not\in A_2$, and so $x\in D$.
Thus there are three types of paths, namely $(v,C,w)$, $(v,A_1,B_1,w)$, and $(v,A_2,D, B_2,w)$.
Paths within each type are edge-disjoint.
Even though $D$ contains $A_1$ and $B_1$, edges between each pair of sets from $\{A_1,B_1,A_2,B_2,C,D,\{v\},\{w\}\}$ occur in at most one of the types, and all edges are between distinct sets from this collection.
Hence no edge is used twice, so all the paths are edge-disjoint. 
The total number of paths is $|C|+|A_1|+|A_2|=d$. 
This finishes the proof if $v$ and $w$ are not adjacent. If $G$ does contain the edge $vw$ then take this as the first path, then remove it and find $d-1$ paths in the same way as above. 
\end{proof}

\begin{corollary}
Let $G$ be a non-collinear visibility graph. 
Then the edge-connectivity of $G$ equals its minimum
degree. Moreover, for distinct vertices $v$ and $w$, there
are $\min\{\deg(v),\deg(w)\}$ edge-disjoint paths of length at most $4$ between $v$ and $w$ in $G$.
\end{corollary}

We now show that not only is the edge connectivity as high as possible, but it is realised by paths with at most one bend in them.

\begin{theorem}\label{nledge}
Let $G$ be a visibility graph with $n$ vertices, at most $\ell$ of which are collinear. 
Then $G$ is $\ceil{\frac{n-1}{\ell-1}}$-edge-connected, which is best possible.
Moreover, between each pair of vertices, there are 
$\ceil{\frac{n-1}{\ell-1}}$ edge-disjoint 1-bend paths. 
\end{theorem}

\begin{proof}
Let $v$ and $w$ be distinct vertices of $G$. 
Let $V^*$ be the set of vertices of $G$ not on the line $vw$.
Let $m := |V^*|$.
Thus $m \geq  n - \ell$.

Let $\mathcal{L}$ be the pencil of lines through $v$ and the vertices in $V^*$.
Let $\mathcal{M}$ be the pencil of lines through $w$ and the vertices in $V^*$.
Let $H$ be the bipartite graph with vertex set
$\mathcal{L}\cup\mathcal{M}$, where $L\in\mathcal{L}$ is adjacent to
$M\in\mathcal{M}$ if and only if $L\cap M$ is a vertex in $V^*$.

Thus $H$ has $m$ edges, and maximum degree at most $\ell-1$.
Hence, by K\"onig's theorem \cite{Konig1916}, $H$ is $(\ell-1)$-edge-colourable.
Thus $H$ contains a matching of at least $\frac{m}{\ell-1}$ edges.
This matching corresponds to a set $S$ of at least $\frac{m}{\ell-1}$ vertices in $V^*$,
no two of which are collinear with $v$ or $w$.

For each vertex $x\in S$, take the path in the visibility graph from
$v$ straight to $x$ and then straight to $w$.
These paths are edge-disjoint.
Adding the path straight from $v$ to $w$, we get at least $\frac{m}{\ell-1} +1$ paths,
which is at least  $\frac{n-1}{\ell-1}$.
Figure~\ref{lbound} shows that this bound is best possible.
\end{proof}

\begin{figure}
\includegraphics{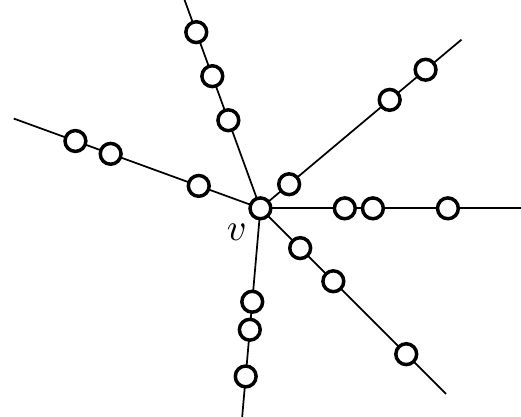}
\caption{\label{lbound}If each ray from $v$ through $V(G)$ contains $\ell$ vertices, the degree of $v$ is $\frac{n-1}{\ell -1}$.}
\end{figure}

We now prove that minimum sized edge cuts in non-collinear visibility graphs are only found around a vertex. To do this, we first characterise the diameter $2$ graphs for which it does not hold.

\begin{proposition}\label{deltacut} 
Let $G$ be a graph with diameter at most $2$ and minimum degree $\delta \geq 2$.
Then $G$ has an edge cut of size $\delta$ that is not the set of edges incident to a single vertex if and only if $V(G)$ can be partitioned into $A\cup B\cup C$ such that:
\begin{itemize}
\item $G[A]\cong K_\delta$ and $|B\cup C|\geq \delta$,
\item each vertex in $A$ has exactly one neighbour in $B$ and no neighbours in $C$,
\item each vertex in $B$ has at least one neighbour in $A$, and
\item each vertex in $B$ is adjacent to each vertex in $C$.
\end{itemize}
\end{proposition}

\begin{proof}
If $G$ has the listed properties then the edges between $A$ and $B$ form a cut of size $\delta$ that is not the set of edges incident to a single vertex.

Conversely, suppose an edge cut of size $\delta$ separates the vertices of $G$ into two sets $X$ and $Y$ with $|X|>1$ and $|Y| >1$.
Each vertex of $X$ is incident to at least $\delta-(|X|-1)$ edges of the cut. 
It follows that $\delta \geq |X|(\delta-(|X|-1))$. 
Consequently, $|X|(|X|-1)\geq \delta(|X|-1)$ and thus $|X|\geq \delta$.
Analogously, $|Y|\geq \delta$.
Since $G$ has diameter $2$, there are no vertices $x \in X$ and $y \in Y$, such that
 all the neighbours of $x$ are in $X$, and all the neighbours of $y$ are in $Y$.
Thus we may assume without loss of generality that all vertices in $X$ have a neighbour in $Y$.
Since there are only $\delta$ edges between $X$ and $Y$, $|X|=\delta$ and each vertex in $X$ has exactly one neighbour in $Y$.
The minimum degree condition implies that all edges among $X$ are present.
Let $A:=X$, $B:= \bigcup_{x\in X}\{N(x) \setminus X \} $ and $C:= V(G) \setminus (A \cup B)$.
If there is a vertex $c \in C$ then $c$ must be joined to all vertices in $B$, otherwise there would be a vertex in $A$ at distance greater than $2$ from $c$. 
\end{proof}

We now prove that diameter $2$ graphs such as those described in Proposition~\ref{deltacut} cannot be visibility graphs. 

\begin{theorem}\label{edgecut} Every minimum edge-cut in a non-collinear visibility graph is the set of edges incident to some vertex.
\end{theorem}

\begin{proof}
Let $G$ be a non-collinear visibility graph. Suppose for the sake of contradiction that $G$ has an edge cut of $\delta(G)$ edges that are not all incident to a single vertex.
Since $G$ is non-collinear, $\delta \geq 2$.
By Proposition~\ref{deltacut}, $V(G)$ can be partitioned into $A \cup B \cup C$ with $|A|=\delta$, $|B \cup C| \geq \delta$, and $\delta$ edges between $A$ and $B$. 
Furthermore, the vertices in $A$ can pairwise see each other and each vertex in $A$ has precisely one neighbour in $B$. 

Choose any $a \in A$ and draw the pencil of $\delta$ rays from $a$ to all other vertices of the graph.
All rays except one contain a point in $A\setminus\{a\}$.
Say two rays are \emph{neighbours} if they bound a sector of angle less than $\pi$ with no other ray inside it. Observe that every ray has at least one neighbour.

First suppose $a$ is in the interior of the convex hull of $V(G)$, as in Figure~\ref{minedgecut}(a).
Then every ray has two neighbours, so each point in $B \cup C$ can see at least one point of $A \setminus \{a\}$ on a neighbouring ray. 
Hence $C$ is empty and $|B|\geq \delta$. 
Along with the edge from $a$ to its neighbour in $B$ we have at least $\delta +1$ edges between $A$ and $B$, a contradiction.

If we cannot choose $a$ in the interior of $\conv(V(G))$, then $A$ is in strictly convex position because no three points of $A$ are collinear.
Let the rays from $a$ containing another point from $A$ be called \emph{$A$-rays}.
The $A$-rays are all extensions of diagonals or edges of $\conv(A)$. 
There is one more ray $r$ that contains only points of $B\cup C$. 
In fact, $r$ has only one point $b$ on it, since all of $r$ is visible from the point in $A$ on a neighbouring ray. Furthermore, the rays which extend diagonals of $\conv(A)$ contain no points of $B\cup C$ since $A$ lies in the boundary of $\conv(V(G))$.
Hence the rest of $B \cup C$ must lie in the two rays which extend the sides of $\conv(A)$.
If these rays both have a neighbouring $A$-ray, then we can argue as before and find $\delta +1$ edges between $A$ and $B$.
We are left with the case where some $A$-ray has $r$ as its only neighbour.
If $b$ lies outside $\conv(A)$ and $\delta >2$ (Figure~\ref{minedgecut}(b)), then we can change our choice of $a$ to a point $a'$ on a ray neighbouring $r$, and then we are back to the previous case. 
(If $\delta =2$ then the other point of $A$ will see $b$ so there can be no more points of $B\cup C$).
Otherwise $b$ is the only point in the interior of $\conv(A)$ (Figure~\ref{minedgecut}(c)), and is therefore the only point in $B$ since it sees all of $A$.
In this case $C$ must be empty since $b$ blocks $c$ from at most one point in $A$.
Thus $|B\cup C| =1$, a contradiction.
\end{proof}

\begin{figure}
\includegraphics{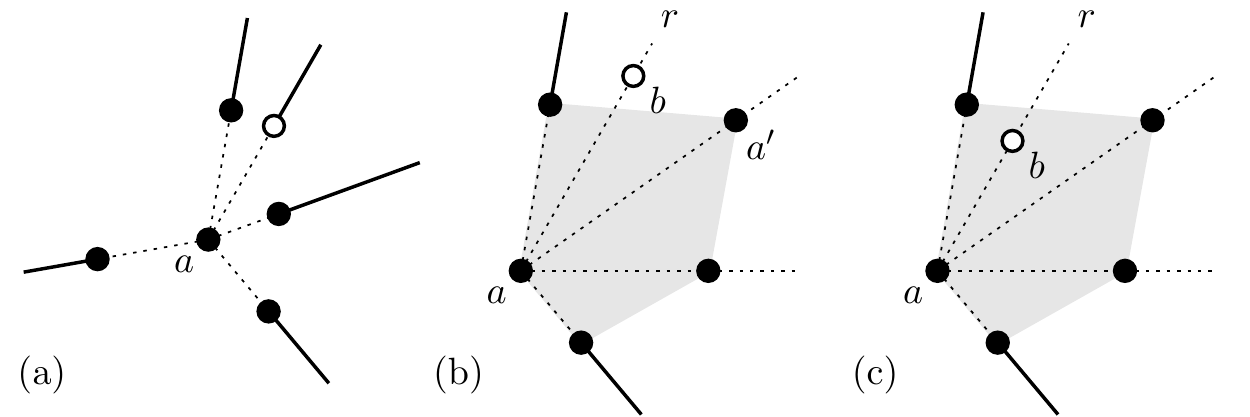}
\caption{\label{minedgecut}In each case the remaining points of $B \cup C$ must lie on the solid segments of the rays.}
\end{figure}

\section{A Key Lemma}

We call a plane graph drawn with straight edges a \emph{non-crossing geometric graph}. 
The following interesting fact about non-crossing geometric graphs will prove useful. It says that two properly coloured non-crossing geometric graphs that are separated by a line can be joined by an edge such that the union is a properly coloured non-crossing geometric graph.
Note that this is false if the two graphs are not separated by a line, as demonstrated by the example in Figure~\ref{counterex}.

\begin{figure}[!h]
\includegraphics{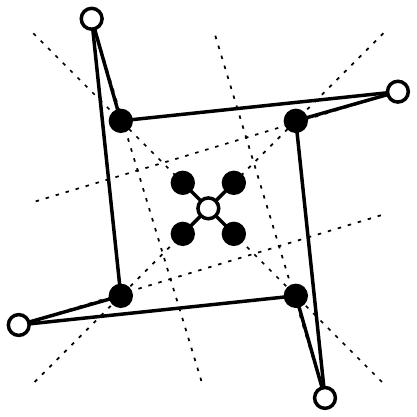}
\caption{\label{counterex}Two properly coloured non-crossing geometric graphs with no black-white edge between them.}
\end{figure}

\begin{lemma}\label{atlem}
Let $G_1$ and $G_2$ be two properly coloured non-crossing geometric graphs with at least one edge each. Suppose their convex hulls are disjoint and that $V(G_1) \cup V(G_2)$ is not collinear.
Then there exists an edge $e \in V(G_1) \times V(G_2)$ such that $G_1 \cup G_2 \cup \{e\}$ is a properly coloured non-crossing geometric graph.
\end{lemma}

\begin{proof}
Let $h$ be a line separating $G_1$ and $G_2$. Assume that $h$ is vertical with $G_1$ to the left. Let $G:=G_1 \cup G_2$.

Call a pair of vertices $v_1 \in V(G_1)$ and $v_2 \in V(G_2)$ a \emph{visible pair} if the line segment between them does not intersect any vertices or edges of $G$.
We aim to find a visible pair with different colours, so assume for the sake of contradiction that every visible pair is monochromatic.

We may assume that $G_1$ and $G_2$ are edge maximal with respect to the colouring, since the removal of an edge only makes it easier to find a bichromatic visible pair.

Suppose the result holds when there are no isolated vertices in $G$. Then, if there are isolated vertices,
we can ignore them and find a bichromatic visible pair $(v_1,v_2)$ in the remaining graph. 
If the edge $v_1v_2$ contains some of the isolated vertices, then it has a sub-segment joining two vertices of different colours. If these vertices lie on the same side of $h$ then the graphs were not edge maximal after all. If they are on different sides, then they are a bichromatic visible pair. Thus we may assume that there are no isolated vertices in $G$.

Let $l$ be the line containing a visible pair $(v_1,v_2)$, then the \emph{height} of the pair is the point at which $l$ intersects $h$.
Call the pair \emph{type-1} if $v_1$ and $v_2$ both have a neighbour strictly under the line $l$ (Figure \ref{atlemma}(a)).
Call the pair \emph{type-2} if there are edges $v_1w_1$ in $G_1$ and $v_2w_2$ in $G_2$ such that 
the line $g$ containing $v_1w_1$ intersects $v_2w_2$ (call this point $x$), 
$w_2$ lies strictly under $g$, 
and the closed triangle $v_1v_2x$ contains no other vertex (Figure \ref{atlemma}(b)). 
Here $x$ may equal $v_2$, in which case $g = l$. A visible pair is also type-2 in the equivalent case with the subscripts interchanged.

A particular visible pair may be neither type-1 nor type-2, but we may assume there exists a type-1 or type-2 pair. To see this, consider the highest visible pair $(v_1,v_2)$ and assume it is neither type-1 nor type-2 (see Figure \ref{atlemma}(c)). Note that $v_1v_2$ is an edge of the convex hull of $G$.
Since all of $G$ lies on or below the line $l$ containing $v_1v_2$, both vertices must have degree $1$ and their neighbours $w_1$ and $w_2$ must lie on $l$. 
For $i=1,2$, let $x_i$ be a vertex of $G_i$ not on $l$ that minimizes the angle $\angle v_iw_ix_i$. Since $V(G)$ is not collinear, at least one of $x_i$ exists. 
By symmetry, we may assume that either only $x_1$ exists, or both $x_1$ and $x_2$ exist and
${\rm dist}(x_1,l)\leq{\rm dist}(x_2,l)$.
In either case, $(x_1,v_2)$ and $(x_1,w_2)$ are visible pairs and at least one of them is bichromatic.


So now assuming there exists a type-1 or type-2 visible pair, let $(u_1,u_2)$ be the lowest such pair:

\emph{Case (i)} 
The pair $(u_1,u_2)$ is type-1 (see Figure \ref{atlemma}(d)).
Let $u_1w_1$ be the first edge of $G_1$ incident to $u_1$ in a clockwise direction, starting at $u_1u_2$.
Let $u_2w_2$ be the first edge of $G_2$ incident to $u_2$ in a counterclockwise direction, starting at $u_2u_1$.
Let $x$ be the point on the segment $u_1w_1$ closest to $w_1$ such that the open triangle $u_1u_2x$ is disjoint from $G$.
Similarly, let $y$ be the point on the segment $u_2w_2$ closest to $w_2$ such that the open triangle $u_1u_2y$ is disjoint from $G$.

Without loss of generality, the intersection of $u_1y$ and $u_2x$ is to the left of $h$, or on $h$.
Therefore the segment $xu_2$ is disjoint from $G_2$. 
Let $v \in V(G_1)$ be the vertex on $xu_2$ closest to $u_2$.
Thus $(v,u_2)$ is a visible pair of height less than $(u_1,u_2)$.
We may assume that $v \neq w_1$, otherwise $(v,u_2)$ would be bichromatic.
The point $w_2$ is under the line $vu_2$ and $v$ has no neighbour above the line $vu_2$.
Hence $v$ either has a neighbour under the line $vu_2$ and $(v,u_2)$ is type-1, or
$v$ has a neighbour on the line $vu_2$ and $(v,u_2)$ is type-2.
This contradicts the assumption that $(u_1,u_2)$ was the lowest pair of either type.

\begin{figure}[t]
\includegraphics{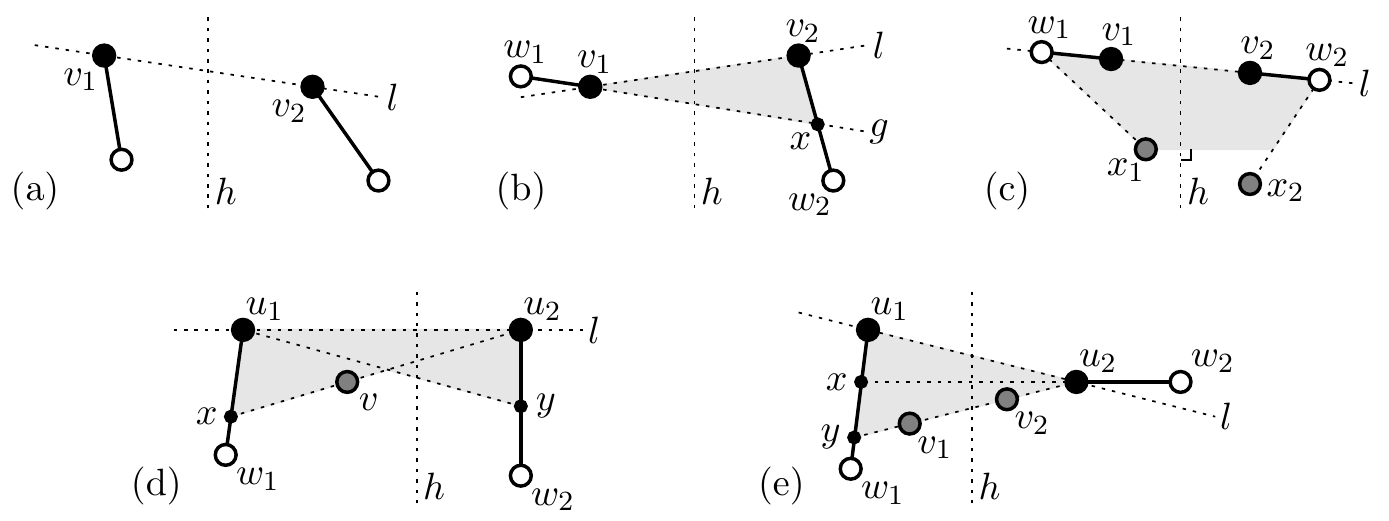}
\caption{\label{atlemma}Proof of Lemma \ref{atlem}. The shaded areas are empty. (a) A type-1 visible pair. (b) A type-2 visible pair. (c) The highest visible pair. (d) The lowest pair is type-1. (e) The lowest pair is type-2. }
\end{figure}

\emph{Case (ii)}
The pair $(u_1,u_2)$ is type-2 with neighbours $w_1$ and $w_2$, such that the line $u_2w_2$ intersects the edge $u_1w_1$ at some point $x$ (see Figure \ref{atlemma}(e)).
Let $y$ be the point on the segment $u_1w_1$ closest to $w_1$ such that the open triangle $u_1u_2y$ is disjoint from $G$. Note $y$ is below $x$.

First assume that $G_2$ intersects $yu_2$. 
Let $v_2$ be the closest vertex to $u_2$ on $yu_2$. 
Thus $(u_1,v_2)$ is a visible pair of height less than $(u_1,u_2)$. 
Let $z$ be a neighbour of $v_2$. 
If $z$ is under the line $u_1v_2$ then $(u_1,v_2)$ is type-1 since $w_1$ is also under this line.
Note $z$ cannot lie above the line $yu_2$ since $u_1$ and $u_2$ see each other and the open triangle $u_1u_2y$ is empty.
Furthermore, if $z$ lies on $yu_2$ then $z=u_2$ and $(u_1,v_2)$ is bichromatic.
Thus if $z$ is not under the line $u_1v_2$, the line $v_2z$ must intersect the edge $u_1w_1$ at a point above $y$, so $(u_1,v_2)$ is a type-2 pair.
Hence the pair $(u_1,v_2)$ is type-1 or type-2, a contradiction.

Now assume that $yu_2$ does not intersect $G_2$, and therefore does intersect $G_1$, and let $v_1 \in V(G_1)$ be the vertex on $yu_2$ closest to $u_2$.
Thus $(v_1,u_2)$ is a visible pair of height less than $(u_1,u_2)$.
We may assume that $v_1 \neq w_1$, otherwise $(v_1,u_2)$ would be bichromatic.
Since $w_2$ is under the line $v_1u_2$,
if $v_1$ has a neighbour under the line $v_1u_2$ then $(v_1,u_2)$ is a type-1 pair.
Otherwise the only neighbour of $v_1$ is on the line $v_1u_2$ which makes $(v_1,u_2)$ a type-2 pair.
Hence the pair $(v_1,u_2)$ is type-1 or type-2, a contradiction.
\end{proof} 

\section{Vertex Connectivity}

As is common practice, we will often refer to vertex-connectivity simply as connectivity. Connectivity of visibility graphs is not as straightforward as edge-connectivity since there are visibility graphs with connectivity strictly less than the minimum degree (see Figure \ref{twothirds}). Our aim in this section is to show that the connectivity of a visibility graph is at least half the minimum degree (Theorem \ref{halfdelta}). This follows from Theorem~\ref{Pavel} below, which says that bivisibility graphs contain large non-crossing subgraphs. In the proof of Theorem~\ref{Pavel} we will need a version of the Ham Sandwich Theorem for point sets in the plane, and also Lemma~\ref{pavlem}.

\begin{theorem}\label{hamsand} (Ham Sandwich. See \cite{matou}.) Let $A$ and $B$ be finite sets of points in the plane. Then there exists a line $h$ such that each closed half-plane determined by $h$ contains at least half of the points in $A$ and at least half of the points in $B$.
\end{theorem}

\begin{lemma}\label{pavlem} Let $A$ be a set of points lying on a line $l$.
Let $B$ be a set of points, none of them lying on $l$.
Let $|A|\geq|B|$.
Then there is a non-crossing spanning tree in the bivisibility graph of $A$ and $B$.
\end{lemma}

\begin{proof}
We proceed by induction on $|B|$. 
If $|B|=1$ then the point in $B$ sees every point in $A$, and we are done.
Now assume $1 < |B| \leq |A| $.

First suppose that all of $B$ lies to one side of $l$ and consider the convex hull $C$ of $A \cup B$. An end point $a$ of $A$ is a corner of $C$ and there is a point $b$ of $B$ visible to it in the boundary of $C$. There exists a line $h$ that separates $\{a,b\}$ from the rest of $A \cup B$. Applying induction and Lemma~\ref{atlem} we find a non-crossing spanning tree among $A\cup B \setminus \{a,b\}$ and an edge across $h$ to the edge $ab$, giving a non-crossing spanning tree of $\B(A,B)$.

Now suppose that there are points of $B$ on either side of $l$. Then we may apply the inductive hypothesis on each side to obtain two spanning trees. Their union is connected, and thus contains a spanning tree.
\end{proof}

\begin{theorem}\label{Pavel} Let $A$ and $B$ be disjoint sets of points in the plane with $|A|=|B|=n$ such that $A\cup B$ is not collinear. Then the bivisibility graph $\B(A,B)$ contains a non-crossing subgraph with at least $n+1$ edges.
\end{theorem}

\begin{proof}
We proceed by induction on $n$. The statement holds for $n=1$, since no valid configuration exists. 
For $n=2$, any triangulation of $A\cup B$ contains at least five edges.
At most one edge has both endpoints in $A$, and similarly for $B$.
Removing these edges, we obtain a non-crossing subgraph of $\B(A,B)$ with at least three edges.
Now assume $n>2$.

\emph{Case (i)} First suppose that there exists a line $l$ that contains at least $n$ points of $A \cup B$.
Let $A_0 := A \cap l$, $B_0 := B \cap l$, $A_1 := A \setminus l$ and $B_1 := B \setminus l$. 
Without loss of generality, $|A_0| \geq |B_0|$. 

If $|A_0|>|B_0|$ then $|A_0|+|B_1|> |B_0| + |B_1|=n$.   
Since $|A_0| + |B_0| \geq n = |B_1| + |B_0|$ we have $|A_0| \geq |B_1|$,
so we may apply Lemma~\ref{pavlem} to $A_0$ and $B_1$. 
We obtain a non-crossing subgraph of $\B(A,B)$ with $|A_0|+|B_1|-1 \geq n$ edges, and 
by adding an edge along $l$ if needed, we are done.

Now assume $|A_0|=|B_0|$. 
We apply Lemma~\ref{pavlem} to $A_0$ and $B_1$, obtaining a non-crossing subgraph with $n-1$ edges, to which we may add one edge along $l$. 
We still need one more edge.
Suppose first that one open half-plane determined by $l$ contains points of both $A_1$ and $B_1$. 
Let $a$ and $b$ be the furthest points of $A_1$ and $B_1$ from $l$ in this half-plane. 
Since $|A_0|=|B_0|$ we may assume that $a$ is at least as far from $l$ as $b$.
Then we may add an edge along the segment $ab$, because none of the edges from $A_0$ to $B_1$ cross it.
It remains to consider the case where $l$ separates $A_1$ from $B_1$. Then applying Lemma~\ref{pavlem} on each side of $l$ we find a non-crossing subgraph with $2n-1$ edges: $|A_0|+|B_1|-1$ on one side, $|B_0|+|A_1|-1$ on the other side, and one more along $l$.

\emph{Case (ii)} Now assume that no line contains $n$ points in $A \cup B$.
By Theorem~\ref{hamsand} there exists a line $h$ such that each of the closed half-planes determined by $h$ contains at least $\frac{n}{2}$ points from each of $A$ and $B$.
Assume that $h$ is horizontal.
Let $A^+$ be the points of $A$ that lie above $h$ along with any that lie on $h$ that we choose to assign to $A^+$. Define $A^-$, $B^+$ and $B^-$ in a similar fashion.
Now assign the points on $h$ to these sets so that each has exactly $\ceil{\frac{n}{2}}$ points. 
In particular, assign the required number of \emph{leftmost} points of $h\cap A$ to $A^+$ and \emph{rightmost} points of $h\cap A$ to $A^-$. Do the same for $h \cap B$ with left and right interchanged.
If $n$ is even then $A^+\cup A^-$ and $B^+\cup B^-$ are partitions of $A$ and $B$.
If $n$ is odd then $|A^+ \cap A^-| = |B^+ \cap B^-| =1$.

Since there is no line containing $n$ points of $A \cup B$, the inductive hypothesis may be applied on either side of $h$. Thus there is a non-crossing subgraph with $\ceil{\frac{n}{2}}+1$ edges on each side.
The union of these subgraphs has at least $n+2$ edges, but some edges along $h$ may overlap. Due to the way the points on $h$ were assigned, one of the subgraphs has at most one edge along $h$. 
(If $n$ is odd, this is the edge between the two points that get assigned to both sides.)
Deleting this edge from the union yields a non-crossing subgraph of $\B(A,B)$ with at least $n+1$ edges.
\end{proof}

\begin{theorem}\label{halfdelta}
Every non-collinear visibility graph with minimum degree $\delta$ has connectivity at least $\frac{\delta}{2} +1$.
\end{theorem}

\begin{proof}
Suppose $\{A,B,C\}$ is a partition of the vertex set of a non-collinear visibility graph such that $C$ separates $A$ and $B$, and $|A|\leq |B|$. By considering a point in $A$ we see that $\delta \leq |A| + |C| -1$. By removing points from $B$ until $|A|=|B|$ whilst ensuring that $A \cup B$ is not collinear, we may apply Theorem~\ref{Pavel}  and Observation \ref{obs} to get $|C|\geq |A| +1$.
Combining these inequalities yields $|C| \geq \frac{\delta}{2}+1$.
\end{proof}

The following observations are corollaries of Theorem \ref{halfdelta}, though they can also be proven directly by elementary arguments.

\begin{proposition}
The following are equivalent for a visibility graph $G$:
(1) $G$ is not collinear,
(2) $\kappa(G)\geq2$,
(3) $\lambda(G)\geq2$ and
(4) $\delta(G)\geq2$.
\end{proposition}


\begin{proposition}
The following are equivalent for a visibility graph $G$:
(1) $\kappa(G)\geq3$,
(2) $\lambda(G)\geq3$ and
(3) $\delta(G)\geq3$.
\end{proposition}


\section{Vertex Connectivity with Bounded Collinearities}

For the visibility graphs of point sets with $n$ points and at most $\ell$ collinear, connectivity is at least $\frac{n-1}{\ell-1}$, just as for edge-connectivity. Bivisibility graphs will play a central role in the proof of this result.
For point sets $A$ and $B$ an \emph{$AB$-line} is a line containing points from both sets.

\begin{theorem} \label{thm:bivis1}
Let $A \cup B$ be a non-trivial partition of a set of $n$ points with at most $\ell$ on any $AB$-line. Then the bivisibility graph $\B(A,B)$ contains a non-crossing forest with at least $\frac{n-1}{\ell-1}$ edges. In particular, if $\ell =2$ then the forest is a spanning tree.
\end{theorem}

\begin{proof} 
The idea of the proof is to cover the points of $A \cup B$ with a large set of disjoint line segments each containing an edge of $G:= \B(A,B)$.
Start with a point $v \in A$. 
Consider all open ended rays starting at $v$ and containing a point of $B$. 
Each such ray contains at least one edge of $G$ and at most $\ell-1$ points of $(A \cup B) \setminus v$.  
For each ray $r$, choose a point $w \in B \cap r$.  
Draw all maximal line segments with an open end at $w$ and a closed end at a point of $A$ in the interior of the sector clockwise from $r$. Figure~\ref{bivisvertl} shows an example. 
 If one sector $S$ has central angle larger than $\pi$ then some points of $A$ may not be covered. 
In this case we bisect $S$, and draw segments from each of its bounding rays into the corresponding half of $S$ (assign points on the bisecting line to one sector arbitrarily). 
Like the rays, these line segments all contain at least one edge of $G$ and at most $\ell-1$ points of $(A \cup B) \setminus \{v,w\}$. 
Together with the rays, they are pairwise disjoint and cover all of $(A \cup B) \setminus v$.
Hence the edges of $G$ contained in them form a non-crossing forest with at least $\frac{n-1}{\ell-1}$ edges.
Note that if $\ell = 2$ we have a forest with $n-1$ edges, hence a spanning tree.
\end{proof}

Note that the $\ell=2$ case of Theorem~\ref{thm:bivis1} is well known \cite{kanekokano}.

\begin{figure}
\includegraphics{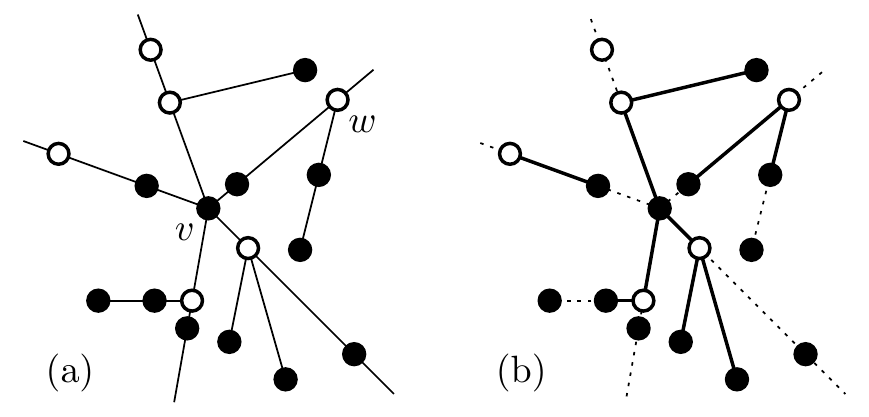}
\caption{\label{bivisvertl}Covering $A\cup B$ with rays and segments (a), each of which contains an edge of the bivisibility graph (b).}
\end{figure}

\begin{corollary}\label{nlvert}
Let $G$ be the visibility graph of a set of $n$ points with at most $\ell$ collinear. Then $G$ has connectivity at least $\frac{n-1}{\ell-1}$, which is best possible.
\end{corollary}

\begin{proof}
Let $\{A,B,C\}$ be a partition of $V(G)$ such that $C$ separates $A$ and $B$.
Consider the bivisibility graph of $A \cup B$.
Applying Observation \ref{obs} and Theorem~\ref{thm:bivis1} (with $n'=n - |C|$ and $\ell' = \ell -1$) yields $|C| \geq \frac{n-|C|-1}{\ell-2}$, which implies $|C| \geq \frac{n-1}{\ell-1}$.
As in the case of edge-connectivity, the example in Figure~\ref{lbound} shows that this bound is best possible.
\end{proof}

In the case of visibility graphs with at most three collinear vertices, it is straightforward to improve the bound in Theorem~\ref{halfdelta}.

\begin{proposition}\label{3line}
Let $G$ be a visibility graph with minimum degree $\delta$ and at most three collinear vertices. Then $G$ has connectivity at least $\frac{2\delta+1}{3}$.
\end{proposition} 

\begin{proof} 
Let $\{A,B,C\}$ be a partition of $V(G)$ such that $C$ separates $A$ and $B$.
Thus each $AB$-line contains only two vertices in $A\cup B$.
Applying Theorem~\ref{thm:bivis1} (with $\ell=2$) and Observation \ref{obs} to $\B(A,B)$ gives $|C| \geq |A| +|B|-1$. For $v \in A$ and $w \in B$ note that $\delta \leq \deg(v) \leq |A| + |C| -1$ and $\delta \leq \deg(w) \leq |B|+|C|-1$. Combining these inequalities gives $|C| \geq \frac{2\delta+1}{3}$.
\end{proof}

In the case of visibility graphs with at most four collinear vertices, the same improvement is found as a corollary of the following theorem about bivisibility graphs. Lemma~\ref{atlem} is an important tool in the proof.

\begin{theorem}\label{tree}
Let $A$ and $B$ be disjoint point sets in the plane with $|A| = |B| = n$ such that $A\cup B$ has at most three points on any $AB$-line. Then the bivisibility graph $\B(A,B)$ contains a non-crossing spanning tree.
\end{theorem}

\begin{proof}
We proceed by induction on $n$. The statement is true for $n=1$. 
Apply Theorem~\ref{hamsand} to find a line $h$ such that each closed half-plane defined by $h$ has at least $\frac{n}{2}$ points from each of $A$ and $B$.
Assume that $h$ is horizontal. 
The idea of the proof is to apply induction on 
each side of $h$ 
to get two spanning trees, and then find an edge joining them together.
In most cases the joining edge will be found by applying Lemma~\ref{atlem}.

We will construct a set $A^+$ containing the points of $A$ that lie above $h$ along with any 
that lie on $h$ that we choose to assign to $A^+$. We will also construct $A^-$, $B^+$ and $B^-$ in a similar fashion.
By the properties of $h$, there exists an 
assignment\footnote{
We need only consider one of the sets, say $A$. Say there are $x$ points above $h$, $y$ points on $h$ and $z$ points below $h$. Then $x+y \geq \ceil{n/2} \geq \floor{n/2} \geq x$ so we can ensure $|A^+| = \ceil{n/2}$. $A^-$ is the complement and therefore has $\floor{n/2}$ points.
}
 of each point in $h \cap (A\cup B)$ to one of these sets such that $|A^+|=|B^+| = \ceil{\frac{n}{2}}$ and $|A^-|=|B^-| = \floor{\frac{n}{2}}$. 

Consider the sequence $s_h$ of signs ($+$ or $-$) given by the chosen assignment of points on $h$ from left to right. 
If $s_h$ is all the same sign, or alternates only once from one sign to the other, then it is possible to perturb $h$ to $h'$ so that $A^+ \cup B^+$ lies strictly above $h'$ and $A^-\cup B^-$ lies strictly below $h'$. 
Thus we may apply induction on 
each side to obtain non-crossing spanning trees in $\B(A^+,B^+)$ and $\B(A^-,B^-)$. Then apply Lemma~\ref{atlem} to find an edge between these two spanning trees, creating a non-crossing spanning tree of $\B(A,B)$.

Otherwise, $s_h$ alternates at least twice (so there are at least three points on $h$). 
This need never happen if there are only points from one set on $h$, since the points required above $h$ can be taken from the left and those required below $h$ from the right. 
Without loss of generality, the only remaining case to consider is that $h$ contains one point from $A$ and two from $B$.
If the two points from $B$ are consecutive on $h$, then without loss of generality $s_h = (+,-,+)$ and the points of $B$ are on the left. In this case the signs of the points from $B$ may be swapped so $s_h$ becomes $(-,+,+)$.
If the point from $A$ lies between the other two points, it is possible that $s_h$ must alternate twice. 
In this case, use induction to find spanning trees in $\B(A^+,B^+)$ and $\B(A^-,B^-)$. These spanning trees have no edges along $h$, so we may add an edge along $h$ to connect them, as shown in Figure~\ref{4linelem}.
\end{proof}

\begin{figure}
\includegraphics{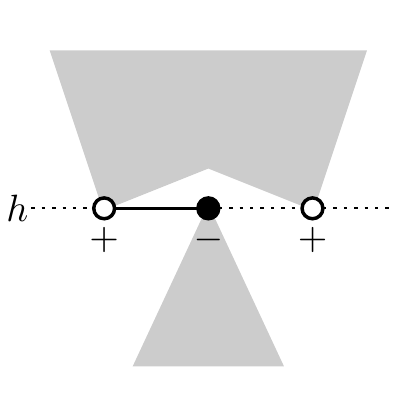}
\caption{\label{4linelem}The only case in which $h$ may not be perturbed to separate the points assigned above $h$ from those assigned below.}
\end{figure}

\begin{theorem}\label{4line}
Let $G$ be a visibility graph with minimum degree $\delta$ and at most four collinear vertices. 
Then $G$ has connectivity at least $\frac{2\delta+1}{3}$.
\end{theorem}

\begin{proof}
Let $\{A,B,C\}$ be a partition of $V(G)$ such that $C$ separates $A$ and $B$ and $|A|\leq |B|$.
By considering a point in $A$ we can see that $\delta \leq |A| + |C| -1$. 
If necessary remove points from $B$ so that $|A|=|B|$.
Applying Theorem~\ref{tree} and Observation \ref{obs} yields $|C|\geq 2|A| -1$.
Combining these inequalities yields $|C| \geq \frac{2\delta+1}{3}$.
\end{proof}

It turns out that Proposition~\ref{3line} and Theorem~\ref{4line} are best possible. There are 
visibility graphs with at most three collinear vertices and connectivity $\frac{2\delta+1}{3}$. The construction was discovered by Roger Alperin, Joe Buhler, Adam Chalcraft and Joel Rosenberg in response to a problem posed by Noam Elkies. Elkies communicated their solution to Todd Trimble who published it on his blog \cite{tvblog}. Here we provide a brief description of the construction, but skip over most background details. Note that the original problem  and construction were not described in terms of visibility graphs, so we have translated them into our terminology. 

The construction uses real points on an elliptic curve.
For our purposes a \emph{real elliptic curve} $\C$ is a curve in the real projective plane (which we model as the Euclidean plane with an extra `line at infinity') defined by an equation of the form $y^2 = x^3 +\alpha x +\beta$. The constants $\alpha$ and $\beta$ are chosen so that the discriminant $\Delta = -16(4\alpha^3+27\beta^2)$ is non-zero, which ensures that the curve is non-singular. We define a group operation `$+$' on the points of $\C$ by declaring that $a+b+c=0$ if the line through $a$ and $b$ also intersects $\C$ at $c$, that is, if $a$, $b$ and $c$ are collinear. The identity element $0$ corresponds to the point at infinity in the $\pm y$-direction, so that for instance $a+b+0=0$ if the line through $a$ and $b$ is parallel to the $y$-axis. Furthermore, $a+a+b=0$ if the tangent line at $a$ also intersects $\C$ at $b$. It can be shown that this operation defines an abelian group structure on the points of $\C$.

We will use two facts about real elliptic curves and the group structure on them. Firstly, no line intersects an elliptic curve in more than three points.
Secondly, the group acts continuously: adding a point $e$ which is close to $0$ (i.e.~very far out towards infinity) to another point $a$ results in a point close to $a$ (in terms of distance along $\C$).

\begin{proposition}\label{ellprop} \emph{(Alperin, Buhler, Chalcraft and Rosenberg)}
For infinitely many integers $\delta$, there is a visibility graph with at most three vertices collinear, minimum degree $\delta$, and connectivity $\frac{2\delta+1}{3}$.
\end{proposition}

\begin{proof}
Begin by choosing three non-zero collinear points $a$, $b$ and $c$ on a real elliptic curve $\C$, such that $c$ lies between $a$ and $b$.
Then choose a point $e$ very close to $0$. Now define
\begin{align*}
A:=& \{ a + ie : 0\leq i \leq m-1 \} \\
B:=& \{ b + je : 0\leq j \leq m-1 \} \\
C:=& \{ -(a+b + ke) : 0 \leq k \leq 2m-2\}.
\end{align*}
Let $G$ be the visibility graph of $A \cup B \cup C$. Since the points are all on $\C$, $G$ has at most three vertices collinear.
Observe that the points $a+ie$ and $b+je$ are collinear with the point $-(a+b+ (i+j)e)$.
Since $e$ was chosen to be very close to $0$, by continuity the set $A$ is contained in a small neighbourhood of $a$, and similarly for $B$ and $C$. Therefore, the point from $C$ is the middle point in each collinear triple, and so $C$ is a vertex cut in $G$, separating $A$ and $B$.

By choosing $a$, $b$ and $c$ away from any points of inflection, we can guarantee that there are no further collinear triples among the sets $A$, $B$ or $C$.
Thus a point in $A$ sees all other points in $A\cup C$, 
a point in $B$ sees all other points in $B\cup C$,
and a point in $C$ sees all other points.
Therefore the minimum degree of $G$ is $\delta = 3m-2$, attained by the vertices in $A\cup B$.
Hence (also using Proposition~\ref{3line}) the connectivity of $G$ is $|C|=2m-1 = \frac{2\delta+1}{3}$.
\end{proof}

\begin{figure}[!h]
\includegraphics{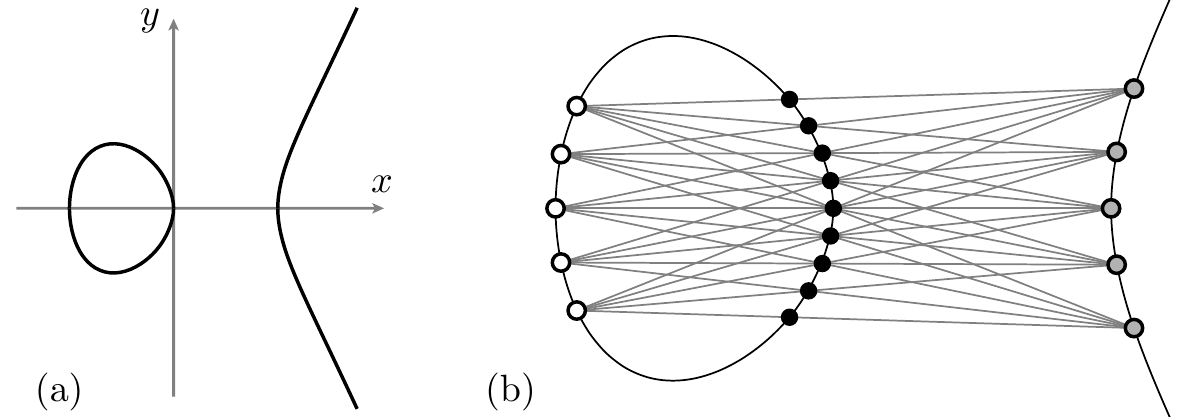}
\caption{\label{elliptic}(a) The elliptic curve $y^2 = x^3 -x$. (b) The black points separate the white points from the grey points.}
\end{figure}

In Figure~\ref{elliptic} we have chosen $\C$ to be the curve $y^2 = x^3 - x$ and the points $a$, $b$ and $c$ on the $x$-axis. We have taken advantage of the symmetry about the $x$-axis to choose $A = \{ a \pm ie \}$ (and similarly for $B$ and $C$), which is slightly different to the construction outlined in Proposition~\ref{ellprop}.

We close our discussion of the connectivity of visibility graphs with the following conjecture.

\begin{conjecture}
Every visibility graph with minimum degree $\delta$ has connectivity at least $\frac{2\delta+1}{3}$.
\end{conjecture}

\section{Connectedness of Bivisibility Graphs}

Visibility graphs are always connected, but bivisibility graphs may have isolated vertices. However, we now prove that non-collinear bivisibility graphs have at most one component that is not an isolated vertex.

\begin{lemma}\label{trilem}
Let $A$ and $B$ be disjoint point sets such that $A\cup B$ is not collinear.
Let $T$ be a triangle with vertices $a \in A$, $b \in B$ and $c \in A \cup B$. Then $a$ or $b$ has a neighbour in $\B(A,B)$ lying in $T \setminus ab$.
\end{lemma}

\begin{proof}
There is at least one point in $T$ not lying on the line $ab$. The one closest to $ab$ sees both $a$ and $b$, and is therefore adjacent to one of them.
\end{proof}

\begin{theorem}
Let $A$ and $B$ be disjoint point sets such that $A\cup B$ is not collinear.
Then $\B(A,B)$ has at most one component
that is not an isolated vertex.
\end{theorem}

\begin{proof}
Assume for the sake of contradiction that $\B(A,B)$ has two components with one or more edges. 
Choose a pair of edges $ab$ and $a'b'$, one from each component, such that the area of $C :=\conv(a,b,a',b')$ is minimal.
If $ab$ and $a'b'$ lie on one line, they are joined by a path through the closest point to that line, a contradiction. 
If they do not lie on a line then both ends of at least one of the edges are vertices of $C$. Assume this
edge is $ab$ and let $v$ be another vertex of $C$ ($v$ is either $a'$ or $b'$). 
Then by Lemma~\ref{trilem}, $a$ or $b$ has a neighbour $w$ in $\triangle abv \setminus ab$. Without loss of generality, $w$ is a neighbour of $a$. If $w = v$, then $ab$ and $a'b'$ are in the same component, a contradiction. If $w\neq v$, then there is a pair of edges with a smaller convex hull, namely $a'b'$ and $aw$, because $w \in C$, but $w$ is not a vertex of $C$. This contradicts the assumption that $C$ was minimal. 
\end{proof}

\begin{corollary}
A non-collinear bivisibility graph is connected if and only if it has no isolated vertices.
\end{corollary}


\bibliographystyle{siam}
\bibliography{references}


\end{document}